\documentclass[preprint,12pt]{elsarticle}

\makeatletter
\def\ps@pprintTitle{%
 \let\@oddhead\@empty
 \let\@evenhead\@empty
 \def\@oddfoot{}%
 \let\@evenfoot\@oddfoot}
\makeatother

\usepackage[utf8]{inputenc}
\usepackage[T1]{fontenc}
\usepackage{mathtools}
\usepackage{amsmath}
\usepackage{xcolor}

\definecolor{verde}{rgb}{0.2,0.6,0.2}
\definecolor{jpurple}{rgb}{0.9,0,0.9}

\usepackage[ruled,linesnumbered]{algorithm2e}
\SetKwInOut{Input}{Input}
\SetKwInOut{Output}{Output}
\usepackage{mathtools,bbold}
\usepackage{multicol}
\usepackage{mathrsfs}
\usepackage{fancyvrb}
\usepackage{algpseudocode}
\usepackage{amsmath,amsthm,amssymb}
\usepackage{dsfont}
\usepackage{booktabs}
\usepackage{siunitx}
\usepackage{float}
\usepackage[shortlabels]{enumitem}
\newtheorem{theorem}{Theorem}[section]
\newtheorem{lemma}[theorem]{Lemma}
\usepackage{hyperref}
\hypersetup{
    colorlinks=true,
    }

\newcommand{\rev}[1]{\textcolor{black}{#1}}

\theoremstyle{definition}

\newtheorem{example}[theorem]{Example}

\usepackage{todonotes}

\theoremstyle{remark}

\numberwithin{equation}{section}

\begin{document}

\title{Constructing Magic Squares: \rev{An} Integer Constraint Satisfaction Problem and a Fast \rev{Approach}}


\author[1]{João Vitor Pamplona\corref{cor1}}
\ead{joao.vitor.pamplona@ufsc.br}

\author[2]{Maria Eduarda Pinheiro}
\ead{mariae.pinheiro18@gmail.com}

\author[1]{Luiz-Rafael Santos}
\ead{l.r.santos@ufsc.br}

\cortext[cor1]{Corresponding author}

\address[1]{Departamento de Matemática, Universidade Federal de Santa Catarina, Campus Blumenau, Blumenau-SC, 89065-300, Brazil}
\address[2]{Departamento de Matemática, Universidade Federal de Santa Catarina,  Florianópolis-SC, 88040-900, Brazil}


  \begin{abstract}
Magic squares are a fascinating mathematical challenge that has intrigued mathematicians for centuries. Given a positive (and possibly large) integer \( n \), one of the main challenges that still remains is to find, within a \rev{reasonable} computational time, a magic square of order \( n \), that is, a square matrix of order \( n \) with unique integers from \( a_{\min} \) to \( a_{\max} \), such that the sum of each row, column, and diagonal equals a constant \( \mathcal{C}(A) \). 
In this work, we first present an integer constraint satisfaction problem  for constructing a magic square of order \( n \). Nonetheless, the solution time of this problem grows exponentially as the order increases. To overcome this limitation, we also propose a \rev{fast approach} that constructs magic squares depending on whether \( n \) is odd, singly even, or doubly even\rev{. Moreover, we provide a proof of the correctness of this novel approach.} Our numerical results show \rev{that our method can  construct} magic squares of order up to \num{70000} in less than \num{140} seconds, demonstrating its efficiency and scalability.
\end{abstract}


\maketitle



\section{Introduction}

Magic Squares are numerical puzzles that have captivated mathematicians and enthusiasts for centuries. While \rev{their} exact origins remain a mystery, a Chinese myth speaks of a turtle rising from the depths of the Lo River during a deluge. Adorning its shell was a peculiar design: a three-by-three grid bearing unique numbers, each row, column, and diagonal summing to fifteen. Mathematically, it can be expressed as the matrix:
\[
\begin{bmatrix}
4 & 9 &  2 \\ 
3 & 5 & 7 \\
8 & 1 & 6
\end{bmatrix}.
\]

There \rev{is other evidence} of Magic Squares not only in China but  India as well, dating back to before our era. For instance, a magic square of order $3$ appears in an 8th-century Arabic manuscript linked to Apollonius of Tyana \citep{berthelot1888collection}.

Around 1530, \citet{von1531occulta} wrote \emph{De Occulta Philosophia}, in which he discussed Magic Squares of order 3 to 9. Those squares were associated with astrological planets: Saturn (associated with order 3), Jupiter (order 4), Mars (order 5), \rev{the Sun} (order 6), Venus (order 7), Mercury (order 8), and \rev{the Moon} (order 9).

Formally, given an integer  $n \geq 3$, and a positive integer $a_{\min},$  one can define a \emph{magic square of order $n$} as a square matrix $A \coloneqq [a_{ij}] \in \mathds{N}^{n \times n}$ with \emph{magic constant}
\[
\mathcal{C}(A)\coloneqq \dfrac{n(n^2+1)}{2} + n(a_{\min}-1), 
\]
that satisfies the following properties:
\begin{enumerate}[(i)]
\item Each component in $A$ is unique.
\item The minimum element in $A$ is $a_{\min}$ and the maximum is $a_{\max} \coloneqq  a_{\min} + n^2 -1 $,
\item  For each $j\in \{1,\ldots, n\}$,  $ \sum_{i=1}^{n} a_{ij}= \mathcal{C}(A)$, \emph{i.e.}, the sum of the elements of each column is $\mathcal{C}(A)$. 
\item  For each $i\in \{1,\ldots, n\}$,  $ \sum_{j=1}^{n} a_{ij} = \mathcal{C}(A) $. Thus, the sum of the elements of each row is equal to $\mathcal{C}(A)$.
\item The trace of $A$ is equal to $\mathcal{C}(A)$, that is, $\sum_{i=1}^n a_{ii} =  \mathcal{C}(A)$.
\item The sum of the secondary diagonal is equal to $\mathcal{C}(A)$, \emph{i.e.}, $$\sum_{i=1}^n a_{i,n-i+1} =  \mathcal{C}(A).$$

\end{enumerate}

More than being a mathematical puzzle, Magic Squares have been used in applications such as cryptography, computer science, and even in the design of games. For instance, \citet{centermass}
 shows that Magic Squares create a very tiny part of square matrices \rev{whose} mass center and their geometrical center coincide. \citet{introductionapll} presents some of the most important unsolved problems  and some physical applications of Magic Squares. Some Linear Algebra \rev{properties} of Magic Squares are also presented in \citet{Lee:2006tw,Essen:1990fd,LOLY20092659}. Besides these properties, Magic Squares are also studied even in pure Algebra, as can be seen in \citet{algebra1,algebra2}.  
 
 In this short note, we present two novel approaches for constructing Magic Squares of arbitrary order. Our first approach frames the problem as a Constraint Satisfaction Problem (CSP). This strategy of reformulating mathematical puzzles as CSPs is well-established in optimization literature. Prior work includes CSP models for edge-matching puzzles \cite{csppuzzle}, the classic Sudoku puzzle \cite{sudokupuzzle2}, and even a variation known as the Jidoku puzzle \cite{sudokupuzzle}.
 
 To the best of our knowledge, the only prior work that formulates the construction of Magic Squares as a CSP model is that of \citet{crawford2009solving}. However, \rev{their approach only considers $a_{min}=1$}.  Moreover, their computational experiments are limited to Magic Squares of order \( 7 \) only, taking 487 seconds to solve.  In contrast, the CSP model presented in this work is applicable to Magic Squares of any order. Although this model always yields a valid Magic Square when a solution is found, it becomes computationally expensive for large values of \( n \).

To overcome this limitation, we introduce a second approach that leverages the divisibility properties of $n$ to efficiently generate Magic Squares. To put our \rev{method} in perspective with existing \rev{approaches}, the work by \citet{heu1} proposed a two-phase heuristic for solving magic squares. However, the authors report that the method proposed by them requires over two hours to construct a magic square of order 500, whereas our \rev{method} completes the task in under \num{0.5} milliseconds; see our numerical experiments below. Moreover,  \citet{heu4} focuses on an algorithm to construct all magic squares of order $4$.
Recent works by \citet{heu2,heu3} introduced multistage evolutionary algorithms for magic square construction. Their experimental results show a solution time of \num{6} minutes for a magic square order 90; meanwhile, our \rev{algorithm} achieves this in under \num{e-4} seconds. 

\rev{Most of the approaches mentioned above do not offer a guarantee of correctness. In contrast, we establish the validity of our fast approach and ensure it consistently produces a magic square. Furthermore, we provide extensive numerical results showing reliability across a wide range of cases, for magic squares of orders as large as $n=\num{70000}$.}

The remainder of this paper is organized as follows. In Section~\ref{sec:CSP}, we present the CSP model for constructing Magic Squares. In Section~\ref{sec:fastapproach}, we introduce our \rev{fast} approach, which is divided into three cases based on the order of the Magic Square. \rev{Section \ref{sec:correctness} is devoted to proving the correctness of this novel approach}.
In Section~\ref{sec:numerical}, we present numerical results comparing the performance of \rev{both approaches (the CSP formulation and the fast \rev{approach}).} Finally, in Section~\ref{sec:conclusion}, we offer concluding remarks and discuss potential future research directions.

\section{Integer Constraint Satisfaction Problem }\label{sec:CSP}
In this section\rev{,} we show that a magic square of order $n$ with magic constant $\mathcal{C}(A)$ can be found by solving a constraint satisfaction problem (CSP) where the variables are binary. More precisely, given  $i,j \in [1,n] := \{1, \cdots, n\}$ and a number  $k \in [a_{\min},a_{\max}],$ the variable $x_{ij}^k$ \rev{takes the following values}
  $$x_{ij}^k =
  \begin{cases}
    1, & \text{if number $k$ is in $\{i,j\}$}, \\
    0, & \text{otherwise}.
  \end{cases}$$

Hence, our goal is to find $x \in \{0,1\}^{n \times n \times n^2}$ that satisfies the following \rev{properties}
\begin{subequations}
\allowdisplaybreaks
\label{MCSP}
\begin{align}
\sum_{i=1}^n \sum_{j=1}^n x_{ij}^k 
&= 1,  
\quad k \in [a_{\min},a_{\max}],
\label{MCSP:1} \\
\sum_{k=a_{\min}}^{a_{\max}} x_{ij}^k 
&= 1, 
\quad i,j \in [1,n],
\label{MCSP:2} \\
\sum_{j=1}^n \sum_{k=a_{\min}}^{a_{\max}} k x_{ij}^k 
&= \mathcal{C}(A), 
\quad i \in [1,n],
\label{MCSP:3} \\
\sum_{i=1}^n \sum_{k=a_{\min}}^{a_{\max}} k x_{ij}^k 
&= \mathcal{C}(A), 
\quad j \in [1,n],
\label{MCSP:4} \\
\sum_{i=1}^n \sum_{k=a_{\min}}^{a_{\max}} k x_{ii}^k 
&= \mathcal{C}(A),
\label{MCSP:5} \\
\sum_{i+j=n+1} \sum_{k=a_{\min}}^{a_{\max}} k x_{ij}^k 
&= \mathcal{C}(A),
\label{MCSP:6} \\
x_{ij}^k 
\in \{0,1\}, 
\quad &i,j \in [1,n], 
\quad k \in [a_{\min}, a_{\max}].
\label{MCSP:7}
\end{align}
\end{subequations}

The rationale of the model is that, for instance, Constraints \eqref{MCSP:1} and \eqref{MCSP:2} ensure, respectively, that  each number can be \rev{assigned to} only one cell and each cell must contain a number $k \in [a_{\min}, a_{\max}]$.  The guarantee that the sum of each row and column of $A$ is equal to $\mathcal{C}(A)$ is given by Constraints~\eqref{MCSP:3} and \eqref{MCSP:4}. Constraints~\eqref{MCSP:5} and \eqref{MCSP:6} enforce that the trace and the sum of the secondary diagonal of $A$ \rev{are} $\mathcal{C}(A)$. Finally, Constraint~\eqref{MCSP:7} \rev{leads} to 
  $$x_{ij}^k =
  \begin{cases}
    1, & \text{if number $k$ is in $\{i,j\}$}, \\
    0, & \text{otherwise},
  \end{cases}$$
for all  $i,j \in [1,n]$, and $ k \in [a_{\min}, a_{\max}].$
Based on that, each \rev{entry} $a_{ij}$ of the magic square $A$ is given by 
$$
a_{ij} = k \in [a_{\min}, a_{\max}]: x_{i,j}^k = 1, \quad i,j \in [1,n].
$$
\rev{We conjecture that the problem is NP-hard, though a formal proof is beyond the scope of this work}. Our focus here is on solution techniques, not a formal complexity analysis\rev{.}

As mentioned earlier, an integer constraint satisfaction problem can be computationally challenging, \rev{a fact that is corroborated}  by the numerical tests presented later in this paper.  

To mitigate this computational burden, in the next section we propose a new algorithm that can efficiently construct Magic Squares of any size. This \rev{algorithm} offers a substantial performance improvement compared to the traditional CSP solver, or other known methods, especially for larger orders.

\section{A Fast Approach}\label{sec:fastapproach}

In this section, we present algorithms for constructing Magic Squares of arbitrary order $n$. Without loss of generality, we consider $a_{\min} \coloneqq 1$. If $a_{\min} \neq 1,$ after solving the problem with $a_{\min} =1$ and getting $a_{ij}$, we can redefine $a_{ij}$ to be $a_{ij} + a_{\min} - 1$.

We categorize our approach into three distinct cases, each requiring a specific algorithmic solution:
\begin{enumerate}
    \item \emph{Odd Order}: When $n$ is odd, we employ the well-established Siamese Method, a technique attributed to \citet{siam};

    \item \emph{Doubly Even Order}: For $n$ divisible by $4$, we introduce a novel algorithm that leverages a specific pattern of number placement within the square;

    \item \emph{Singly Even Order}: When $n$ is even but not divisible by $4$, we propose a novel \rev{approach} that systematically fills the square, ensuring the magic square sum property.
\end{enumerate}

\subsection{Odd Magic Squares}
\label{subsec:hindu}

The Siamese method, also known as the De La Loubere method, \rev{attributed to \citet{siam}}, is a well-known algorithm for constructing odd-order magic squares. It is based on the principle of placing numbers in a specific pattern to ensure that the sum of each row, column, and diagonal is equal. The algorithm works as follows: \rev{the numbers are distributed in order}, from $1$ to $n^2$, in the positions of the matrix $A \coloneqq [a_{ij}] \in \mathds{N}^{n \times n}$, with $n$ odd. Recall that $a_{\min} = 1$ and $a_{\max} = n^2$. The steps to achieve the desired magic square are \rev{summarized} next:

\begin{enumerate}
\item The integer $1$ is placed in the central position of the first \rev{row}; that is, \rev{set} $a_{1,(n+1)/2} = 1$.

\item For any $k=1,\ldots,n^2-1$ already allocated, $k+1$ will be allocated in the position that remains:
\begin{enumerate}
 
  \item on the upper right diagonal, if \rev{it exists};
  \item in the last row of the column to \rev{the} right, if there is no upper row;
  \item in the first column of the top row, if there is no column on the right;

\item in the same column but in the \rev{row below} if the top right diagonal position is already occupied or if $k$ is allocated in $a_{1,n}$.

\end{enumerate}

\end{enumerate}

\rev{The correctness of the Siamese method is demonstrated in \citet{reference_siam}. Specifically, for any odd $n$, the matrix $A$ generated by the Siamese method is a magic square with the magic constant $\mathcal{C}(A).$}

\subsection{Doubly Even Magic Squares}\label{subsec:double}

We now delineate our \rev{approach} for constructing a Magic Square of order $n$ when $n$ is doubly even, meaning that $n$ is divisible by $4$. The approach is based on the following steps, which are \rev{detailed in} Algorithm~\ref{algo:doubleeven}. Again, recall we consider $a_{\min} = 1$ and $a_{\max} = n^2$.

\begin{algorithm}[H]
  \caption{\rev{Doubly Even Magic Square}}
 \label{algo:doubleeven}
   \SetKwInput{Input}{Input~}
  \Input{$n$ and $a_{\min}$}

 Consider the following division of blocks for matrix $A$ as shown \rev{below}:
\[
A=\begin{bmatrix}
X & Y & Y & X \\
Y & X & X & Y \\
Y & X & X & Y \\
X & Y & Y & X
\end{bmatrix},
\]
where both blocks of type $X$ and $Y$ are of order $\Big(\dfrac{n}{4} \times \dfrac{n}{4}\Big)$.  \label{step1}

 Compute  $a_ {ij}$ as:
\[
a_{ij} = \begin{cases}
(i-1)n + j, & \text{if }a_{ij} \in X, \\ 
n^2 + 1 - (n(i-1) + j), & \text{if }a_{ij}\in Y.
\end{cases}
\] 
\label{step2}
  \end{algorithm}
The key to this novel approach is the partitioning of the matrix into blocks, two of which are filled with numbers in a straightforward manner, while the other two blocks are \rev{filled} with their complements. This ensures that the sum of each row, column, and diagonal remains constant.

We \rev{now present a small example} to show how the algorithm works.

\begin{example}
Consider $n=4$, \emph{i.e.}, the simplest example for \rev{a} doubly even Magic Square. Then, following Step~\ref{step1}, we have the following partition of the matrix $A$:
\begin{align*}
    a_{11}, a_{14}, a_{22}, a_{23}, a_{32}, a_{33}, a_{41}, a_{44} & \in X, \\ 
    a_{12}, a_{13}, a_{21}, a_{24}, a_{31}, a_{34}, a_{42}, a_{43} & \in Y.
\end{align*}
Now, Step~\ref{step2} yields the following:
\[
a_{ij} = \begin{cases}
(i-1)4 + j, & \text{if }a_{ij}\in X, \\ 
17 - (4(i-1) + j), & \text{if }a_{ij}\in Y.
\end{cases}
\]
Hence,
\begin{multicols}{2}
\begin{itemize}
\item $a_{11}$ = $(1-1)4 + 1$ = 1
\item $a_{12}$ = $17 - ((1-1)4 + 2)$ = 15
\item $a_{13}$ = $17 - ((1-1)4 + 3)$ = 14
\item $a_{14}$ = $(1-1)4 + 4$ = 4
\item $a_{21}$ = $17 - ((2-1)4 + 1)$ = 12
\item $a_{22}$ = $(2-1)4 + 2$ = 6
\item $a_{23}$ = $(2-1)4 + 3$ = 7
\item $a_{24}$ = $17 - ((2-1)4 + 4)$ = 9
\item $a_{31}$ = $17 - ((3-1)4 + 1)$ = 8
\item $a_{32}$ = $(3-1)4 + 2$ = 10
\item $a_{33}$ = $(3-1)4 + 3$ = 11
\item $a_{34}$ = $17 - ((3-1)4 + 4)$ = 5
\item $a_{41}$ = $(4-1)4 + 1$ = 13
\item $a_{42}$ = $17 - ((4-1)4 + 2)$ = 3
\item $a_{43}$ = $17 - ((4-1)4 + 3)$ = 2
\item $a_{44}$ = $(4-1)4 + 4$ = 16
\end{itemize}
\end{multicols}
Finally, the resultant matrix is given by 
\[
A=\begin{bmatrix}
1 & 15 & 14 & 4 \\
12 & 6 & 7 & 9 \\
8 & 10 & 11 & 5 \\
13 & 3 & 2 & 16
\end{bmatrix}.
\]

Note that the matrix above is in fact a 
magic square \rev{of} order $4$ with $a_{\min} =1$, with $\mathcal{C}(A) = 34$. 
\end{example}

\subsection{Singly Even Magic Squares}\label{subsec:singleeven}

Our last proposal makes use of the algorithm for an odd Magic Square to construct a Magic Square of order $n$ when $n$ \rev{is} singly even; that is, we consider now $n$ an even number that is not \rev{divisible} by $4$. \rev{Again, without loss of generality, we consider $a_{\min}=1$.} Algorithm~\ref{algo:singleeven} describes the steps to construct the desired Magic Square.

\begin{algorithm}[H]
  \caption{Singly Even Magic Square}
 \label{algo:singleeven}
   \SetKwInput{Input}{Input~}
  \Input{$n$ and $a_{\min}$}

Consider    
$
A=\begin{bmatrix}
A^1 & A^3 \\
A^4 & A^2 \\
\end{bmatrix},
$ 
with $A^1, \cdots, A^4 $ as follows: \\ 
  \For{ $\ell = \{1,  \dotsc, 4\}$}
  {Create $A^\ell$ the Magic Square of odd order \rev{ using the Siamese method (see Section~\ref{subsec:hindu})}, with $a^\ell_{\min} = (\ell - 1) \dfrac{n^2}{4} +1$ the minimum \rev{entry} of $A^\ell$.\label{algo:step3}\\ 
  \uIf{$\ell \in \{1,4\}$ \label{step4}}{Consider the following partition of $A^\ell$
 \[
 A^{\ell}=\begin{bmatrix}
A^{\ell}_1 & x_{\ell} & A^{\ell}_{2} \\
y_{\ell}^\top & w_{\ell}^\top & \alpha_{\ell} \\
A^{\ell}_3 & z_{\ell} & A^{\ell}_4
\end{bmatrix}, 
\]
with $A^{\ell}_r \in \mathds{N}^{m \times m},$ for $r\in [1,4] $, $  \quad x_\ell, w_\ell, y_\ell, z_\ell \in \mathds{N}^m, \quad \alpha_\ell \in \mathds{N}$ and $m \coloneqq \frac{n-2}{4}.$
  \label{algo:step5}
  }
  \Else{Consider the following partition of $A^\ell$ 
 \[  
 A^{\ell}=\begin{bmatrix}
 \rev{B^{\ell}_1} &  \rev{B^{\ell}_2}
\end{bmatrix},
\]
where $ \rev{B^{\ell}_1}  \in \mathds{N}^{\frac{n}{2} \times \frac{(n+6)}{4}}$ and  $ \rev{B^{\ell}_2}  \in \mathds{N}^{\frac{n}{2} \times \frac{(n-6)}{4}}$.  \label{algo:step7}}
    }
Exchange the following sub-blocks
\begin{multicols}{2}
\begin{enumerate}

   \item $A^1_1 \leftrightarrow A^4_1$
\item $A^1_3 \leftrightarrow A^4_3$
\item $w_1 \leftrightarrow w_4$
\item $\rev{B^3_2} \leftrightarrow \rev{B^2_2}$
\end{enumerate}
\end{multicols}
\label{algo:step10}
Return the matrix $A$ .
\end{algorithm}

We show next how the algorithm \rev{works} for a simple case.

\begin{example}
  
Consider $n=10$. First, we create the submatrices $A^1, \cdots, A^4$, as described in Step~\ref{algo:step3}:
\begin{center}
A = \begin{tabular}{|ccccc|ccccc|}
\cline{1-10} 
17 & 24 & 1 & 8 & 15 & 67 & 74 & 51 & 58 & 65\\
23 & 5 & 7 & 14 & 16 & 73 & 55 & 57 & 64 & 66\\
4 & 6 & 13 & 20 & 22 & 54 & 56 & 63 & 70 & 72\\
10 & 12 & 19 & 21 & 3 & 60 & 62 & 69 & 71 & 53\\
11 & 18 & 25 & 2 & 9 & 61 & 68 & 75 & 52 & 59\\
\cline{1-10} 
92 & 99 & 76 & 83 & 90 & 42 & 49 & 26 & 33 & 40 \\
98 & 80 & 82 & 89 & 91 & 48 & 30 & 32 & 39 & 41\\
79 & 81 & 88 & 95 & 97 & 29 & 31 & 38 & 45 & 47\\
85 & 87 & 94 & 96 & 78 & 35 & 37 & 44 & 46 & 28\\
86 & 93 & 100 & 77 & 84 & 36 & 43 & 50 & 27 & 34  \\ 
\cline{1-10}
\end{tabular}
\end{center}

Then, we partition the sub-matrix as described in Steps~\ref{algo:step5} and \ref{algo:step7}.

\begin{center}
A = \begin{tabular}{|ccccc|ccccc|}
\cline{1-10} 
\textcolor{green}{17} & \textcolor{green}{24} & 1 & 8 & 15 & 67 & 74 & 51 & 58 & \textcolor{purple}{65}\\
\textcolor{green}{23} & \textcolor{green}{5} & 7 & 14 & 16 & 73 & 55 & 57 & 64 & \textcolor{purple}{66}\\
4 & 6 & \textcolor{blue}{13} & \textcolor{blue}{20} & 22 & 54 & 56 & 63 & 70 & \textcolor{purple}{72}\\
\textcolor{red}{10} & \textcolor{red}{12} & 19 & 21 & 3 & 60 & 62 & 69 & 71 & \textcolor{purple}{53}\\
\textcolor{red}{11} & \textcolor{red}{18} & 25 & 2 & 9 & 61 & 68 & 75 & 52 & \textcolor{purple}{59}\\
\cline{1-10} 
\textcolor{green}{92} & \textcolor{green}{99} & 76 & 83 & 90 & 42 & 49 & 26 & 33 & \textcolor{purple}{40}\\
\textcolor{green}{98} & \textcolor{green}{80} & 82 & 89 & 91 & 48 & 30 & 32 & 39 & \textcolor{purple}{41}\\
79 & 81 & \textcolor{blue}{88} & \textcolor{blue}{95} & 97 & 29 & 31 & 38 & 45 & \textcolor{purple}{47}\\
\textcolor{red}{85} & \textcolor{red}{87} & 94 & 96 & 78 & 35 & 37 & 44 & 46 & \textcolor{purple}{28}\\
\textcolor{red}{86} & \textcolor{red}{93} & 100 & 77 & 84 & 36 & 43 & 50 & 27 & \textcolor{purple}{34} \\
\cline{1-10}
\end{tabular}
\end{center}

Finally, we permute the sub-blocks as described in Step~\ref{algo:step10} (in this example sub-blocks with the same color are permuted).

\begin{center}
A = \begin{tabular}{|ccccc|ccccc|}
\cline{1-10} 
\textcolor{green}{92} & \textcolor{green}{99} & 1 & 8 & 15 & 67 & 74 & 51 & 58 & \textcolor{purple}{40}\\
\textcolor{green}{98} & \textcolor{green}{80} & 7 & 14 & 16 & 73 & 55 & 57 & 64 & \textcolor{purple}{41}\\
4 & 6 & \textcolor{blue}{88} & \textcolor{blue}{95} & 22 & 54 & 56 & 63 & 70 & \textcolor{purple}{47}\\
\textcolor{red}{85} & \textcolor{red}{87} & 19 & 21 & 3 & 60 & 62 & 69 & 71 & \textcolor{purple}{28}\\
\textcolor{red}{86} & \textcolor{red}{93} & 25 & 2 & 9 & 61 & 68 & 75 & 52 & \textcolor{purple}{34}\\
\cline{1-10} 
\textcolor{green}{17} & \textcolor{green}{24} & 76 & 83 & 90 & 42 & 49 & 26 & 33 & \textcolor{purple}{65}\\
\textcolor{green}{23} & \textcolor{green}{5} & 82 & 89 & 91 & 48 & 30 & 32 & 39 & \textcolor{purple}{66}\\
79 & 81 & \textcolor{blue}{13} & \textcolor{blue}{20} & 97 & 29 & 31 & 38 & 45 & \textcolor{purple}{72}\\
\textcolor{red}{10} & \textcolor{red}{12} & 94 & 96 & 78 & 35 & 37 & 44 & 46 & \textcolor{purple}{53}\\
\textcolor{red}{11} & \textcolor{red}{18} & 100 & 77 & 84 & 36 & 43 & 50 & 27 & \textcolor{purple}{59} \\
\cline{1-10}
\end{tabular}
\end{center}

As expected, $A$ as described above is a Magic Square with order $10$ and $a_{\min} = 1$. 
\end{example}

\section{Correctness of our novel approaches}\label{sec:correctness}

\rev{
In this section, we establish the correctness of Algorithms~\ref{algo:doubleeven} and \ref{algo:singleeven}, which are designed to construct magic squares of doubly even and singly even orders, respectively.} 

\subsection{Algorithm~\ref{algo:doubleeven}}
\rev{\rev{We begin} by proving the correctness of Algorithm~\ref{algo:doubleeven}, that is, that our method yields a valid magic square for all doubly even values of $n$. We first establish a lemma concerning the symmetry of the summations.}

\begin{lemma}\label{lema1}
\rev{Let $n$ be a doubly even integer, and define the outer and inner index sets as $\mathcal{B}_{\rev{\text{out}}} := \{1, \dots, \frac{n}{4}\} \cup \{\frac{3n}{4}+1, \dots, n\}$ and $\mathcal{B}_{\rev{\text{in}}} := \{\frac{n}{4}+1, \dots, \frac{3n}{4}\}$. The sum of the integers in the middle two quarters is equal to the sum of the integers in the first and last quarters. That is:
\begin{equation}\label{sum}
\sum_{j \in \mathcal{B}_{\rev{\text{out}}} } j = \sum_{j\in\mathcal{B}_{\rev{\text{in}}} }  j.
\end{equation}
}
\end{lemma}
    \begin{proof}
  
\rev{Consider $n = 4k$ for some $k\in \mathds{N}$. Using the arithmetic series formula $\sum_{j=a}^{b} j = \frac{(b-a+1)(a+b)}{2},$ we evaluate both sides of \eqref{sum}. The left-hand side, consisting of the first and last quarters of the range, yields:
\[\sum_{j \in \mathcal{B}_{\rev{\text{out}}} } j =  \sum_{j=1}^{k} j + \sum_{j=3k+1}^{4k} j = \frac{k(k+1)}{2} + \frac{k(7k+1)}{2} = 4k^2 + k. \]
Similarly, the right-hand side comprises the middle two quarters:
 \[\sum_{j \in \mathcal{B}_{\rev{\text{in}}} } j = \sum_{j=k+1}^{3k} j = \frac{2k(4k+1)}{2} = 4k^2 + k. \]
That is, the identity holds for any doubly even $n$.}
\end{proof}

\rev{We now show that the sum of each row of $A$ produced by Algorithm~\ref{algo:doubleeven} is equal to the magic constant $\mathcal{C}(A).$} 
\begin{lemma}\label{rowdouble}
    \rev{Assuming $n$ is doubly even, Algorithm~\ref{algo:doubleeven} produces a matrix $A$ where each row sum is equal to the magic constant $\mathcal{C}(A) = \dfrac{n(n^2+1)}{2}$.}
\end{lemma}
\begin{proof}
    \rev{As a consequence of the block decomposition in Step~\ref{step1} and the definition of the elements $a_{ij}$ in Step~\ref{step2},  for the rows in the outer index set, $i\in \mathcal{B}_{\rev{\text{out}}} $, we \rev{fix} $i$ and apply Lemma~\ref{lema1} to obtain }
    $$
    \begin{array}{rcl}
       \rev{  \displaystyle\sum_{j=1}^n a_{ij}}   & \rev{=}&  \rev{  \displaystyle\sum_{j\in \mathcal{B}_{\rev{\text{out}}}} \rev{(}(i-1)n + j\rev{)} + \displaystyle\sum_{j\in \mathcal{B}_{\rev{\text{in}}}}\rev{(}n^2 + 1 - (n(i-1) + j)\rev{)}} \\
       &\rev{=} & \rev{\dfrac{n}{2}((i-1)n) + \rev{\frac{n}{2}}(n^2 +1) -\rev{\frac{n}{2}}(n(i-1)) + \displaystyle\sum_{j\in \mathcal{B}_{\rev{\text{out}}}}  j - \displaystyle\sum_{j\in \mathcal{B}_{\rev{\text{in}}}}  j}\\ 
         &\rev{=} & \rev{\dfrac{n(n^2+1)}{2} = \mathcal{C}(A)}.
    \end{array}
    $$
\rev{The case for $i \in \mathcal{B}_{\rev{\text{in}}}$ is strictly analogous to the one above. Therefore, for every row $i \in \{1, \dots, n\}$, the sum of the elements is $\mathcal{C}(A)$.}
\end{proof}

\rev{We now prove that the sum of the elements in each column of the matrix generated by Algorithm~\ref{algo:doubleeven} is also $\dfrac{n(n^2+1)}{2}$.}

\begin{lemma}\label{columdouble}
  \rev{  Given that $n$ is \rev{doubly even}, the matrix $A$ generated by Algorithm~\ref{algo:doubleeven} satisfies the column sum property, where each sum equals the magic constant:
\[ \mathcal{C}(A) = \frac{n(n^2+1)}{2} .\]}
\end{lemma}

\begin{proof}
    \rev{In view of the block decomposition in Step~\ref{step1} and the definition of the entries $a_{ij}$ in Step~\ref{step2}, for any column index $j \in \mathcal{B}_{\rev{\text{out}}}$, we \rev{fix} $j$ and using Lemma~\ref{lema1} obtain:}
    $$
    \begin{array}{rcl}
       \rev{  \displaystyle\sum_{i=1}^n a_{ij}}   & \rev{=}&  \rev{  \displaystyle\sum_{\rev{i}\in \mathcal{B}_{\rev{\text{in}}}} \rev{(}(i-1)n + j\rev{)} + \displaystyle\sum_{\rev{i}\in \mathcal{B}_{\rev{\text{out}}}}\rev{(}n^2 + 1 - (n(i-1) + j)\rev{)}} \\
       &\rev{=} & \rev{\dfrac{n}{2}(\rev{-n}+j + n^2 +1 +n-j) + n\left(\displaystyle\sum_{\rev{i}\in \mathcal{B}_{\rev{\text{out}}}}  i - \displaystyle\sum_{\rev{i}\in \mathcal{B}_{\rev{\text{in}}}}  i\right)}\\ 
         &\rev{=} & \rev{\dfrac{n(n^2+1)}{2} = \mathcal{C}(A)}.
    \end{array}
    $$
\rev{As with the row cases \rev{in} Lemma~\ref{rowdouble}, the proof for column indices $j \in \mathcal{B}_{\rev{\text{in}}}$ is analogous to the one above. Hence, for each column of $A$ the sum of the elements is $\mathcal{C}(A).$}

\end{proof}
\rev{Using the Lemmas above, the following theorem establishes the correctness of Algorithm~\ref{algo:doubleeven}.}
\begin{theorem}

\rev{Assuming $n$ is doubly even, Algorithm~\ref{algo:doubleeven} produces a magic square matrix with magic constant $\mathcal{C}(A) = \dfrac{n(n^2+1)}{2}$.}

    \end{theorem}

\begin{proof}
    \rev{In view of Lemmas~\ref{rowdouble} and \ref{columdouble}, it remains only to show that the sums of the entries along both the main and secondary diagonals are equal to $\mathcal{C}(A)\rev{.}$}
    \rev{Note that the main and secondary diagonals only have elements in the $X$ blocks. Hence\rev{,}}
     $$
    \begin{array}{rcccl}
       \rev{  \displaystyle\sum_{i=1}^n a_{ii}}   & \rev{=}&  \rev{   \displaystyle\sum_{i=1}^n \rev{(}(i-1)n + i\rev{)}} &\rev{=}& \rev{-  n^2  +\displaystyle\sum_{i=1}^n (n+1)i} \\
       &\rev{=} & \rev{-n^2  + (n+1) \frac{\rev{n}(n+1) }{2}}
         &\rev{=} & \rev{\frac{n}{2}(-2n +(n+1)^2)} \\ 
            &\rev{=} &  \rev{\frac{n}{2}(-2n+ n^2+2n+1)}
         &\rev{=} & \rev{\frac{n}{2}(n^2+1)} \\
       &\rev{ = }& \rev{\mathcal{C}(A),} & \\ 
    \end{array}
    $$
\rev{and}
 $$
    \begin{array}{rcccl}
       \rev{  \displaystyle\sum_{i=1}^n a_{i,n-i+1}}   & \rev{=}&  \rev{   \displaystyle\sum_{i=1}^n \rev{(}(i-1)n + n-i+1\rev{)}} &\rev{=}& \rev{n(-n+n+1) +\displaystyle\sum_{i=1}^n (n-1)i} \\
       &\rev{=} & \rev{n  + (n-1) \frac{\rev{n}(n+1) }{2}}
         &\rev{=} & \rev{\frac{n}{2}(2 +(n+1)(n-1)\rev{)}} \\ 
            &\rev{=} & \rev{\frac{n}{2}(2 +n^2-1)} 
         &\rev{=} & \rev{\frac{n}{2}(n^2+1) } \\
       &\rev{ = } &\rev{\mathcal{C}(A).} & \\ 
    \end{array}
    $$
\rev{
That is, both diagonals sum to $\mathcal{C}(A)$, proving that for a doubly even $n$, Algorithm~\ref{algo:doubleeven} generates a magic square.}
\end{proof}

\subsection{Algorithm~\ref{algo:singleeven}}\label{subsec:correctnesssingleeven}

\rev{The design of Algorithm~\ref{algo:singleeven} is grounded in the principle of \textit{offset compensation} across quadrants. Therefore, to establish its correctness, the following lemmas are essential.}

\begin{lemma}\label{lema_until3}
\rev{Let $n$ be a singly even integer. The matrix $A$ constructed up to Step~\ref{algo:step3} of Algorithm~\ref{algo:singleeven} satisfies the following properties:}
\begin{enumerate}
    \item[(P1)] \rev{Every \rev{column} sum is equal to the magic constant $\mathcal{C}(A)$.}
    \item[(P2)] \rev{Each of the first $n/2$ rows has a sum of $\mathcal{C}(A) - \frac{n^3}{8}$.}
    \item[(P3)] \rev{Each of the last $n/2$ rows has a sum of $\mathcal{C}(A) + \frac{n^3}{8}$.}
    \item[(P4)] \rev{The main diagonal sum is $\mathcal{C}(A) - \frac{n^3}{4}$.}
    \item[(P5)] \rev{The secondary diagonal sum is $\mathcal{C}(A) + \frac{n^3}{4}$.}
\end{enumerate}
\end{lemma}

\begin{proof}
\rev{In Step \ref{algo:step3}, by constructing each sub-matrix $A^\ell$ as an odd-order magic square via the Siamese method, we ensure that each quadrant possesses an internal magic constant $\mathcal{C}(A^\ell)$, given by:
\begin{equation*}
    \mathcal{C}(A^\ell) = \frac{p(p^2 + 1)}{2} + p \cdot (a^\ell_{\min}-1) 
\end{equation*}
\rev{where $p = n/2$ and}
\begin{equation*}
    a_{\min}^1 = 1, \quad a_{\min}^2 = \frac{n^2}{4} + 1, \quad a_{\min}^3 = \frac{2n^2}{4} + 1,\quad   a_{\min}^4 = 3\frac{n^2}{4} + 1.
\end{equation*}}

\rev{Using these facts, we can now prove the properties above:
\begin{itemize}
    \item[(P1)] \rev{For the columns in the left half, the sum is given by the sum of the magic constants of the vertical quadrants, $\mathcal{C}(A^1) + \mathcal{C}(A^4)$. Substituting the values for $a^1_{\min}$ and $a^4_{\min}$ and rearranging the terms\rev{,} we obtain:}
    $$
  \begin{array}{rcl}
  \mathcal{C}(A^1) + \mathcal{C}(A^4) &=& \dfrac{p(p^2 + 1)}{2} + p(a^1_{\min}-1) + \dfrac{p(p^2 + 1)}{2} + p(a^4_{\min}-1) \\
  &=& p(p^2 + 1) + p\left( 2+ 3\dfrac{n^2}{4}-2\right) \\
  &=& \dfrac{n}{2}\left(\dfrac{n^2}{4} + 1\right) + \dfrac{3n^3}{8} = \dfrac{n^3}{8} + \dfrac{n}{2} + \dfrac{3n^3}{8}  \\
  &=& \dfrac{n(n^2 + 1)}{2} = \mathcal{C}(A).
\end{array} 
$$
An analogous derivation applies to the right half of the columns, where $\mathcal{C}(A^3) + \mathcal{C}(A^2) = \mathcal{C}(A)$. Thus, by construction, every column of the composite matrix $A$ satisfies the magic constant property.
\item[(P2)] \rev{A row in the upper half sums to $\mathcal{C}(A^1) + \mathcal{C}(A^3)$\rev{:}}
$$
 \begin{array}{rcl}
  \mathcal{C}(A^1) + \mathcal{C}(A^3) &=& \dfrac{p(p^2 + 1)}{2} + p(a^1_{\min}-1) + \dfrac{p(p^2 + 1)}{2} + p(a^3_{\min}-1) \\
  &=& p(p^2 + 1) + p\left( \dfrac{n^2}{2}\right) = \dfrac{n}{2}\left(\dfrac{n^2}{4} + 1 +\dfrac{n^2}{2}\right) \\
  &=& \dfrac{n\left(\dfrac{3n^2}{4} + 1\right)}{2} = \dfrac{n(n^2 + 1)}{2} - \dfrac{n}{2} \dfrac{n^2}{4 } \\ &=& \mathcal{C}(A) - \dfrac{n^3}{8}.
\end{array} 
$$
\item[(P3)] Given a row in the lower half, the sum of all elements in this row satisfies $\mathcal{C}(A^4) + \mathcal{C}(A^2)$, and using the second row of the expression above, we obtain
$$
 \begin{array}{rcl}
  \mathcal{C}(A^2) + \mathcal{C}(A^4) 
  &=& p(p^2 + 1) + p\left( n^2 \right) = \mathcal{C}(A^1) + \mathcal{C}(A^3) + p\left(\dfrac{n^2}{2} \right)   \\
  &=&  \mathcal{C}(A) - \dfrac{n^3}{8} + \dfrac{n^3}{4} =  \mathcal{C}(A) + \dfrac{n^3}{8}.
\end{array} 
$$
\item[(P4)] \rev{Observe that \rev{the} main diagonal sum is $\mathcal{C}(A^1) + \mathcal{C}(A^2)$, and, again using the second row of expression in (P2)\rev{,} we obtain\rev{:}}
$$
 \begin{array}{rcl}
  \mathcal{C}(A^1) + \mathcal{C}(A^2) 
  &=& p(p^2 + 1) + p\left( \dfrac{n^2}{4} \right) = \mathcal{C}(A^1) + \mathcal{C}(A^3) - p\left(\dfrac{n^2}{4} \right)   \\
  &=&  \mathcal{C}(A) - \dfrac{n^3}{8} - \dfrac{n^3}{8} =  \mathcal{C}(A) - \dfrac{n^3}{4}.
\end{array} 
$$
\item[(P5)] The sum of the secondary diagonal is $\mathcal{C}(A^3) + \mathcal{C}(A^4)$. And, using the first equality of expression in (P4)\rev{,} we obtain\rev{:}
$$
 \begin{array}{rcl}
  \mathcal{C}(A^4) + \mathcal{C}(A^3) 
  &=& p(p^2 + 1) + p\left( \dfrac{5n^2}{4} \right)\\ &=& \mathcal{C}(A^1) + \mathcal{C}(A^2) + p\left(n^2 \right)   \\
  &=&  \mathcal{C}(A) - \dfrac{n^3}{4} + \dfrac{n^3}{2} =  \mathcal{C}(A) + \dfrac{n^3}{4}.
\end{array} 
$$
\end{itemize}}
\end{proof}

\rev{\rev{Given} Lemma~\ref{lema_until3}, it remains to be shown that the permutations executed in Steps \ref{step4}--\ref{algo:step10} satisfy three critical conditions: first, that the sum of the columns of $A$ is maintained; second, that the exchange process effectively redistributes the numerical mass by adding $n^3/8$ \rev{to the rows of the first half} and subtracting it \rev{from} the second half; and finally, that the diagonals are corrected by adding $n^3/4$ to the main diagonal while simultaneously removing the same amount from the anti-diagonal. This redistribution ensures that all row and diagonal sums \rev{equal} $\mathcal{C}(A)$.}

\rev{Note that the exchanges in Step~\ref{algo:step10} are strictly vertical. Consequently, the column membership of each element remains unchanged, ensuring that the column sums established in Lemma~\ref{lema_until3} are invariant under these operations. To facilitate the verification of the remaining conditions regarding the row and diagonal sums, we present the following lemmas.}

\begin{lemma}\label{lema38}
\rev{Assuming $n$ \rev{is} a singly even natural \rev{number}, the exchanges $A^1_1 \leftrightarrow A^4_1$, $A^1_3~\leftrightarrow~A^4_3$, and $w_1 \leftrightarrow w_4$ in Step~\ref{algo:step10} of Algorithm~\ref{algo:singleeven} add a total value of $c \coloneqq \frac{3n^3- 6n^2}{16}$ to each row in the upper half of the matrix, while simultaneously removing the same value $c$ from each row in the lower half.}
\end{lemma}

\begin{proof}
\rev{By Step~\ref{algo:step3}, each quadrant $A^\ell$ is a magic square with a minimum element $a^\ell_{\min} = \frac{(\ell - 1)n^2}{4} + 1$. Consequently, the elements of $A^\ell$ satisfy the following offset relationship:
\begin{equation}\label{formaij}
    a_{ij}^\ell = a_{ij}^1 + \dfrac{(\ell -1)n^2}{4}, \quad \forall \ell \in [1,4].
\end{equation}}

\rev{Specifically, the relationship between the values in $A^4$ and $A^1$ is given by $a^4_{ij} = a^1_{ij} + \frac{3n^2}{4}$. Thus, when the exchanges 1, 2, and 3 in Step~\ref{algo:step10} occur, each element from $A^4$ moved to the upper half adds the offset value of $\frac{3n^2}{4}$ to its respective row. Conversely, each element from $A^1$ moved to the lower half results in the removal of $\frac{3n^2}{4}$ from those row sums.} 

\rev{Because the blocks $A^\ell_r \in \mathds{N}^{\frac{n-2}{4} \times \frac{n-2}{4}}$ \rev{and} $w^\ell \in\mathds{N}^{\frac{n-2}{4}} $ for $\ell \in [2,3] $ and $r\in[1,4]$\rev{,} we permute exactly $\frac{n-2}{4}$ elements in each row. That is, the total numerical mass \rev{added} to each row in the upper half is $\frac{3n^2}{4}\times \frac{n-2}{4}= c$, with the same amount $c$ being removed from each row in the lower half.}
\end{proof}

\rev{It remains to determine the impact of the fourth exchange in Step~\ref{algo:step10} on the row sums.}

\begin{lemma}\label{lema39}
\rev{Consider $n$ to be a singly even natural \rev{number}. The exchange $B^3_2 \leftrightarrow B^2_2$ presented in Step~\ref{algo:step10} of Algorithm~\ref{algo:singleeven} removes a total value of $d \coloneqq \frac{n^3-6n^2}{16}$ \rev{from} each row in the upper half of the matrix, while simultaneously adding the same value $d$ \rev{to} each row in the lower half.}
\end{lemma}

\begin{proof}
\rev{By Expression~\eqref{formaij}, the relationship between the values in $A^3$ and $A^2$ is given by $a^3_{ij} = a^1_{ij} + \frac{2n^2}{4}$ and $a^2_{ij} = a^1_{ij} + \frac{n^2}{4}$. Thus, $a^3_{ij} = a^2_{ij} + \frac{n^2}{4}$. This means that when the exchange $B^3_2 \leftrightarrow B^2_2$ is performed, for each exchanged element in the rows of the upper half, we remove the offset value of $\frac{n^2}{4}$. \rev{Analogously}, for each exchanged element in the rows of the lower half, we add $\frac{n^2}{4}$.}

\rev{Since $B^\ell_2 \in \mathds{N}^{\frac{n}{2} \times \frac{(n-6)}{4}}$, for $\ell \in [2,3],$ we permute exactly $\frac{n-6}{4}$ elements per row\rev{; thus,} the total numerical mass \rev{removed from} each row in the upper half is $\frac{n^2}{4} \times \frac{n-6}{4} =d$, with the same amount $d$ being added \rev{to} each row in the lower half.}
\end{proof}

\rev{After the analysis of the \rev{rows}, it remains to investigate how these exchanges redistribute the numerical mass along the diagonals.}

\begin{lemma}\label{lema310}
\rev{The matrix $A$ generated by Algorithm~\ref{algo:singleeven} satisfies the property that both the main and secondary diagonal sums are equal to the magic constant $\mathcal{C}(A)$.}
\end{lemma}

\begin{proof}
\rev{Following the logic established in the previous lemmas, we track the displacement of numerical mass along the diagonals. For the main diagonal, the exchanges $A^1_1 \leftrightarrow A^4_1$ and $w_1 \leftrightarrow w_4$ affect exactly $k = \frac{n-2}{4} + 1 = \frac{n+2}{4}$ elements. Each of these elements $a_{ii}$ is replaced by an element with an offset larger by $\frac{3n^2}{4}$. Consequently, this portion of the main diagonal sum increases by:
\begin{equation*}
    \left(\frac{n+2}{4}\right) \times \frac{3n^2}{4} = \frac{3n^3 + 6n^2}{16}.
\end{equation*}}
\rev{Additionally, the exchange $B^3_2 \leftrightarrow B^2_2$ involves $\frac{n-6}{4}$ elements on the main diagonal, which are increased by $\frac{n^2}{4}$. Referring back to Property (P4) of Lemma~\ref{lema_until3}, the final sum is:
\[
\begin{array}{rcl}
    \displaystyle\sum_{i=1}^n a_{ii} &=&  \mathcal{C}(A) - \dfrac{n^3}{4} + \dfrac{3n^3 + 6n^2}{16} + \dfrac{n-6}{4}\times \dfrac{n^2}{4} \\
    &=& \mathcal{C}(A) - \dfrac{n^3}{4} + \dfrac{3n^3 + 6n^2}{16} + \dfrac{n^3-6n^2}{16} \\
    &=& \mathcal{C}(A) + \dfrac{-4n^3 + 4n^3}{16} = \mathcal{C}(A).
\end{array}
\]}

\rev{Regarding the secondary diagonal, the first $\frac{n+2}{4}$ elements $a_{i, n-i+1}$ are affected by the exchanges $A^1_1 \leftrightarrow A^4_1$ and $w_1 \leftrightarrow w_4$, which decrease the sum of these entries by $\frac{3n^2}{4}$ per element. Furthermore, the exchange $B^3_2 \leftrightarrow B^2_2$ involves $\frac{n-6}{4}$ elements on the secondary diagonal that are decreased by $\frac{n^2}{4}$.}

\rev{Analogous to the main diagonal, we obtain:
\[
\displaystyle sum_{i=1}^n a_{i,n-i+1} = \mathcal{C}(A) + \dfrac{n^3}{4} - \dfrac{3n^3 + 6n^2}{16} - \dfrac{n-6}{4} \times \dfrac{n^2}{4} = \mathcal{C}(A).
\]}
\rev{Thus, both diagonal sums converge to the magic constant.}
\end{proof}

\rev{With all these lemmas at \rev{hand}, we can prove the main theorem.}

\begin{theorem}
\rev{When $n$ is singly even, Algorithm~\ref{algo:singleeven} produces a magic square matrix with magic constant $\mathcal{C}(A) = \dfrac{n(n^2+1)}{2}$.}
\end{theorem}

\begin{proof}
\rev{By Property (P1) of Lemma~\ref{lema_until3} and the fact that the exchanges in Step~\ref{algo:step10} are strictly vertical, the sum of each column in matrix $A$, generated by Algorithm~\ref{algo:singleeven}, remains $\mathcal{C}(A)$. Regarding the diagonals, Lemma~\ref{lema310} establishes that both diagonal sums equal $\mathcal{C}(A)$. We shall now address the row sums.}

\rev{The result of \rev{Lemma} \ref{lema38} combined with \rev{properties} (P2) and (P3) of Lemma~\ref{lema_until3}, demonstrate that after performing the first three exchanges in Step~\ref{algo:step10} of Algorithm~\ref{algo:singleeven}, the row sums are significantly adjusted. Specifically, each of the first $n/2$ rows now possesses a sum of:}
\begin{equation*}
\mathcal{C}(A) - \frac{n^3}{8} + \frac{3n^3-6n^2}{16} = \mathcal{C}(A) + \frac{n^3-6n^2}{16},
\end{equation*}
\rev{while the rows in the lower half (the last $n/2$ rows) conversely result in a sum of $\mathcal{C}(A) - \frac{n^3-6n^2}{16}$. }

\rev{Lemma~\ref{lema39} shows that the fourth exchange in Step~\ref{algo:step10} corrects this remaining imbalance. Hence, the sum of each row is also $\mathcal{C}(A)$ and the matrix $A$ is a magic square.}
\end{proof}

\section{Numerical experiments}\label{sec:numerical}
In this section, we present and discuss our computational results \rev{for constructing} magic squares with the two approaches presented in this paper:

\begin{itemize}
    \item CSP as given in Problem~\eqref{MCSP}, solved with \texttt{Gurobi} \citep{gurobi};
    \item \emph{Fast Approach}: As described in Section~\ref{sec:fastapproach}.
\end{itemize}
We use \texttt{Julia 1.9} \citep{BezansonEdelmanKarpinskiShah17} and the \rev{tests \rev{were}} conducted on a laptop with an Intel Core i9-13900HX processor with a clock speed of 5.40 GHz and 64 GB of RAM.
For any $n$, \rev{we consider} $a_{\min}=1$. The codes are fully available at 
\url{https://github.com/JoaoVitorPamplona/Magic_Square}. Note that the CSP model can also be solved using  open-source solvers, like CBC \cite{forrest2005cbc}. Even though there are other solvers for solving constraint satisfaction problems, since ours involves binary variables, we prefer to use Gurobi, which solves the problem even without an objective function.

The first comparison is the run time of the two approaches. Table~\ref{table1} shows that The CSP model's execution time grows exponentially \rev{as the order of the magic square increases}. This is expected and happens because there is a direct correlation between the value of $n$ and the quantity of binary variables. Table~\ref{table1} also shows that our \emph{Fast Approach} \rev{constructs} a magic square of order $n \leq 11$ in microseconds. 

\begin{table}[H]
    \caption{Time(s) comparison}
\label{table1}
    \centering
    \begin{tabular}{ccc}
        \toprule
        \textbf{Size} & \textbf{Fast Approach} & \textbf{CSP} \\
        \midrule
        3  & $\num{1.2e-6}$ & 0.01 \\
        4  & $\num{3.9e-6}$ & 0.05 \\
        5  & $\num{1.1e-6}$ & 0.06 \\
        6  & $\num{2.5e-6}$ & 0.46 \\
        7  & $\num{1.1e-6}$ & 0.26 \\
        8  & $\num{5.2e-6}$ & 16.75 \\
        9  & $\num{1.9e-6}$ & 29.89 \\
        10 & $\num{5.0e-6}$ & 55.16 \\
        11 & $\num{1.4e-6}$  & 180.95 \\
        \bottomrule
    \end{tabular}
\end{table}

Using our \emph{Fast Approach}\rev{,} we also \rev{constructed} magic squares of order $n~\in~[3,5000]$. As it is \rev{composed} of three distinct algorithms, we can evaluate their execution times:
\begin{itemize}
    \item \emph{Odd}: When $n$ is odd, the Siamese method as discussed in Subsection~\ref{subsec:hindu}.
    \item \emph{Doubly even}: If $n$ is doubly even, the approach proposed in Subsection~\ref{subsec:double}.
   \item \emph{Singly even}: If $n$ is even, but not doubly even, as proposed in Subsection~\ref{subsec:singleeven}.
\end{itemize}
For each order $n$, the magic square is constructed three times, and the minimum execution time is recorded using \texttt{Benchmark Tools} \citep{revels2021benchmarktools}. 

Figure \ref{fig19} presents the execution times, demonstrating that our approach can construct a magic square of order up to $5000$ in under 0.3 seconds. For each of the three distinct algorithms detailed previously, a \rev{quadratic function is fitted} to approximate the pairs (Magic square size, Time(s)).

\begin{figure}[H]
\centering
\includegraphics[scale=.50]{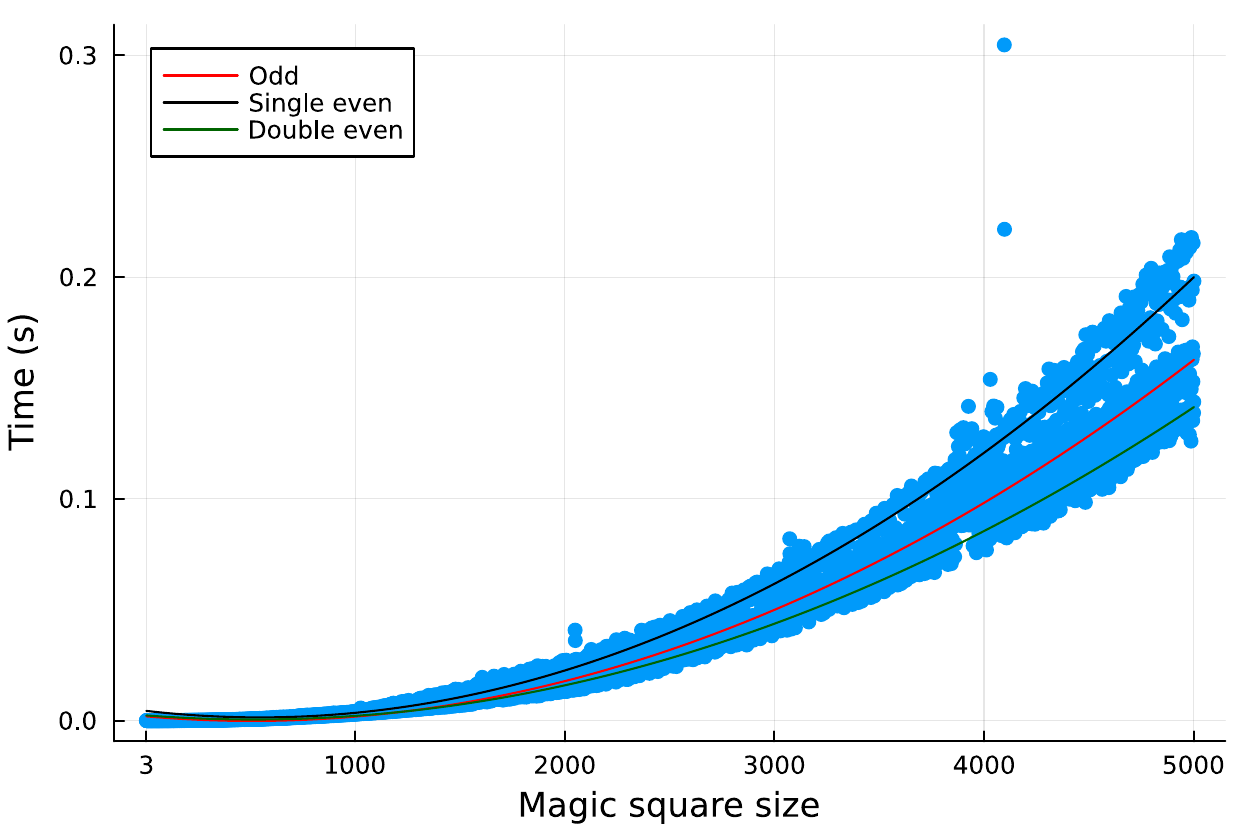}
\caption{Performance of  \emph{Fast Approach} for $n\in [3,5000]$.}
\label{fig19}
\end{figure}

Due to the high speed of our \emph{Fast Approach}\rev{,} we also \rev{computed} magic squares of higher orders. Indeed, for each $p\in [1,7]$ we \rev{computed} the magic squares of orders  
$p \times \num{10000}$, $p \times \num{10000}+1$ and $p \times \num{10000}+2$ once, and \rev{recorded} the time. \rev{These values of $n$ are chosen} to ensure that we are computing magic squares of the three different types of orders. We note that Julia is a just-in-time compiled language. In this case, the first execution is not considered.

Figure~\ref{fig20} shows these execution times as well as the approximation by quadratic functions. Notably, the singly even algorithm exhibits the highest computational complexity among the three. Note that our approach \rev{built} a magic square of order \num{70002} in less than \num{140} seconds, while the CSP Model requires \num{180} seconds to build a magic square of \rev{order 11}. \rev{This is a} notable difference, which highlights the robust performance of our approach, even with \rev{large} values of $n$. We would like to stress that a \num{70000} $\times$ \num{70000} \rev{matrix} consumes \rev{a substantial amount of} memory, so we \rev{did} not test our approach for larger values of $n$. Nevertheless, we believe that our proposed method can be used to construct magic squares with bigger magnitudes, as long as the memory suffices, for time complexity seems to be polynomial, as suggested by the quadratic approximation shown in these experiments.

\begin{figure}[H]
\centering
\includegraphics[scale=.50]{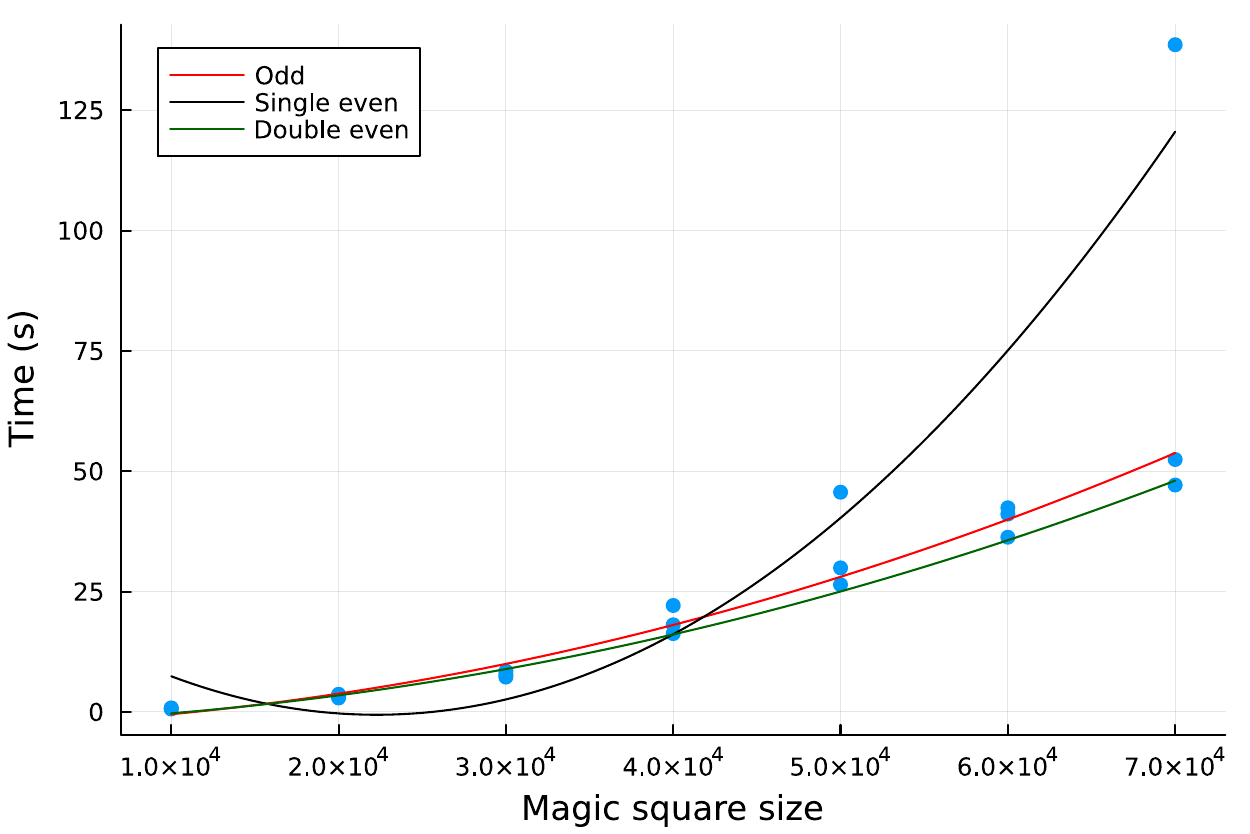}
\caption{Performance of \emph{Fast Approach} for large value of $n$.  }
\label{fig20}
\end{figure}

\section{Concluding remarks}\label{sec:conclusion}
In this paper, we proposed two approaches to construct a magic square of order \( n \). The first approach is a Constraint Satisfaction Problem model, which guarantees a valid magic square. The second algorithm, called the \emph{Fast Approach}, utilizes the classical Siamese method when the order \( n \) is odd. For even orders, we incorporate two novel techniques, applied depending on whether the order is doubly even or singly even. These techniques rely on simple operations, such as element swaps. \rev{We also provide a proof of correctness for the proposed approach.} Numerical results indicate that the computational time required for the CSP formulation increases rapidly as the order grows, as expected. In contrast\rev{,} \rev{our fast approach} constructed magic squares for all \( n \) between \num{3} and \num{5000} in under \num{0.3} seconds, and for \( n \) around \num{70000} in less than \num{140} seconds. Future research could explore \rev{the potential extension of the proposed approach} to related combinatorial structures, \rev{such as Sudoku or the $n$-Queens problem.}
\rev{}

\section*{Conflicts of interest}
All authors certify that they have no affiliations with or involvement in any organization or entity with any financial interest or non-financial interest in the subject matter or materials discussed in this manuscript.

\section*{Acknowledgments}
\rev{The authors would like to thank the anonymous reviewers for their valuable comments and suggestions that helped improve the quality of this paper. Part of this research was conducted while \textbf{JVP} was a doctoral researcher at the Department of Economics and \textbf{MEP} was a postdoctoral researcher at the Department of Mathematics, both at Trier University. \textbf{JVP} was partially funded by the German Federal Ministry of Education and Research (BMBF) through the OptimAgent Project'' and by the Deutsche Forschungsgemeinschaft (DFG) within the Research Training Group RTG 2126, Algorithmic Optimization''. \textbf{MEP} was also supported by the DFG RTG 2126 and by the Brazilian agency Conselho Nacional de Desenvolvimento Científico e Tecnológico (CNPq, Grant 314788/2025-5). \textbf{LRS} was partially funded by CNPq (Grants 310571/2023-5 and 407147/2023-3) and by the Fundação de Amparo à Pesquisa e Inovação do Estado de Santa Catarina (FAPESC, Edital 21/2024, Grant 2024TR002238).}

\bibliographystyle{plainnat}  
\bibliography{bibtex}

\end{document}